\documentclass[submission]{FPSAC2020}
\usepackage{etex}
\usepackage{amssymb}
\usepackage[T1]{fontenc}
\usepackage[latin1]{inputenc}
\usepackage{mathrsfs}
\usepackage{amsfonts,amsmath,lmodern,fancyhdr,lastpage,graphicx,nccfoots,caption}
\usepackage{afterpage}
\usepackage{tikz}
\usepackage{pgf,tikz}
\usepackage{placeins}
\usepackage{amsthm}
\usepackage{graphicx}
\usepackage{latexsym}
\usepackage{amsmath}
\usepackage{amsfonts}
\usepackage{amssymb}
\usepackage{pstricks-add}
\usepackage{slashbox}
\usepackage{url}
\usepackage{multicol}
\usepackage{xcolor}

\newcounter{a}
\ifodd\thea\else\stepcounter{a}\fi

\newcommand{\cchi}{\mbox{\raise.48ex\hbox{\,$\chi$}}}

\newcommand{\R}{\mathbb{R}}
\newcommand{\C}{\mathbb{C}}
\newcommand{\N}{\mathbb N}
\newcommand{\Z}{\mathbb Z}

\newcommand{\YT}[3]{
	\vcenter{\hbox{
			\begin{tikzpicture}[x={(0in,-#1)},y={(#1,0in)}] % matrix coordinate
			\foreach \rowi [count=\i] in {#3} {
				\foreach \e [count=\j] in \rowi {
					\draw (\i,\j) rectangle +(-1,-1);
					\draw (\i-0.5,\j-0.5) node {$#2\e$};
				}
			}
			\end{tikzpicture}
		}}
	}
	
\newcommand{\SYT}[3]{
	\vcenter{\hbox{
			\begin{tikzpicture}[x={(0in,-#1)},y={(#1,0in)}] % matrix coordinate
			\foreach \rowi [count=\i] in {#3} {
				\foreach \e [count=\j] in \rowi {
					\draw (\i,-\j) rectangle +(-1,-1);
					\draw (\i-0.5,-\j-0.5) node {$#2\e$};
				}
			}
			\end{tikzpicture}
		}}
	}

\newtheorem*{question*}{Question}
\newtheorem{theorem}{Theorem}
\newtheorem{corollary}[theorem]{Corollary}

\newtheorem{lemma}[theorem]{Lemma}
\newtheorem{proposition}[theorem]{Proposition}
\newtheorem*{remark}{Remark}
\newtheorem{example}[theorem]{Example}

\theoremstyle{definition}
\newtheorem*{algorithm}{Algorithm}
\newtheorem{definition}[theorem]{Definition}

\title{Symplectic keys and Demazure atoms in type $C$}
\author[J. M. Santos]{Jo\~{a}o Miguel Santos
\thanks{\href{mailto:jmsantos@mat.uc.pt}{jmsantos@mat.uc.pt}. This work was partially supported by the Center for Mathematics of the  University of Coimbra - UID/MAT/00324/2019, funded by the Portuguese Government through FCT/MEC and co-funded by the European Regional Development Fund through the Partnership Agreement PT2020. It was also supported by FCT, through the grant PD/BD/142954/2018, under POCH funds, co-financed by the European Social Fund and Portuguese National Funds from MEC.}}
\address{CMUC, Department of Mathematics, University of Coimbra, Apartado 3008,
	3001--454 Coimbra, Portugal}%\email{jmsantos@mat.uc.pt}
\keywords{Keys, Demazure crystal graph, Demazure characters and  atoms in type C.}

\abstract{We compute, mimicking the Lascoux-Schützenberger type $A$ combinatorial  procedure, left and right keys for a Kashiwara-Nakashima tableau in type $C$. These symplectic keys have a role similar to the keys for semistandard Young tableaux. More precisely, our symplectic keys give a tableau criterion for the Bruhat order on the hyperoctahedral group and cosets, and describe Demazure atoms and characters in type $C$. The right and the left symplectic keys are related through the Lusztig involution. A type $C$ Schützenberger evacuation is defined to realize that involution.}

\resume{Nous calculons, en imitant la procédure combinatoire du type $A$ de Lascoux-Schützenberger, les clés gauche et droite pour un tableau de Kashiwara-Nakashima du type $C$. Ces clés symplectiques ont un rôle similaire aux clés des tableaux de Young semistandard. Plus précisément, nos clés symplectiques fournissent un critère de tableau pour l'ordre de Bruhat sur le groupe hyperoctaédrique et ses classes, et décrivent les atomes et les caractères de Demazure du type $C$. Les clés symplectiques droite et gauche sont liées par l'involution de Lusztig. Une évacuation de Schützenberger du type $C$ est définie pour réaliser cette involution.}
\begin{document}
\maketitle
%\bigskip
\section{Introduction}
%The irreducible characters of the general linear group over $\mathbb{C}$, the Schur functions, are combinatorially expressed as sums on semistandard Young tableaux \cite{You 52}. When restricting to the symplectic group, 
To generate the characters of a given finite dimensional irreducible representation of the symplectic Lie algebra $sp(2n, \C)$, two different types of symplectic tableaux have been proposed: the King tableaux \cite{King 75} and the De Concini tableaux \cite{DeC 79}.
%Sheats  has exhibited a bijection between them  \cite{She 99}.  
We work with symplectic Kashiwara and Nakashima tableaux, %\cite{Kash 95} 
 which are a variation of  De Concini tableaux, and with its crystal structure. That crystal structure allows a plactic monoid compatible with insertion and sliding algorithms, and Robinson-Schensted type correspondence, studied by Lecouvey in terms of crystal isomorphisms \cite{Lec 02}. %The generalization of the notion of plactic monoid for finite Cartan types was first introduced by Littelmann using his path model \cite{Litt 96}. Symplectic  Kashiwara-Nakashima tableaux are the ones that we work with, in the  corresponding ambient plactic monoid. 
%We  however note that very recently  Lee    has endowed  King tableaux with a crystal structure \cite{Lee 19}.

Kashiwara \cite{Kash 93} and Littelmann \cite{ Litt 95} have shown that Demazure characters \cite{Dem 74}, for any Weyl group, can be lifted to certain subsets of Kashiwara-Nakashima tableaux, called Demazure crystals. Demazure characters (key polynomials) are then generated over Demazure crystals. In type $C_n$, they are non symmetric Laurent polynomials, with respect to the action of the Weyl group, which can be seen as "partial" symplectic characters. 
%The Demazure crystal, $\mathfrak{B}_v$ where $v\in \Z^n$, in the Weyl group orbit of a partition $\lambda$, is a union of disjoint sets, \emph{Demazure atoms}, $\widehat{\mathfrak{B}}_u$ where $u\in \Z^n$,  over an interval in the Bruhat order, on the classes modulo the stabilizer of $\lambda$. This order, induced on the orbit of $\lambda$, gives $\mathfrak{B}_v=\displaystyle \biguplus_{\lambda\le u\le v} \widehat{\mathfrak{B}}_u$. 
Given a partition $\lambda$, let $v$ be in the orbit of $\lambda$ under the action of the Weyl group, the Demazure crystal, $\mathfrak{B}_v$, is a union of disjoint sets, \emph{Demazure crystal atoms}, $\widehat{\mathfrak{B}}_u$,  over an interval in the Bruhat order, on the cosets modulo the stabilizer of $\lambda$. This order, induced on the orbit of $\lambda$, gives $\mathfrak{B}_v=\displaystyle \biguplus_{\lambda\le u\le v} \widehat{\mathfrak{B}}_u$.

In type $A_{n-1}$, Lascoux and Sch\"{u}tzenberger \cite{LasSchu 90} identified tableaux with nested columns as key tableaux, and defined the right key map that sends tableaux to key tableaux. Their right key map can be used to describe the type $A$ Demazure atoms $\widehat{\mathfrak{B}}_u,\,u\in \N^n$ \cite[Theorem 3.8]{LasSchu 90}.
%gives a decomposition  of  $\mathfrak{B}^\lambda$ into non intersecting subsets, in bijection with the vectors $v$ in the orbit of $\lambda$, under the action of the Weyl  group \cite[Theorem 3.8]{LasSchu 90}.
%They have called \emph{standard bases} to the sum of monomial weights over $\mathfrak{U}(v)$, which, after Mason \cite{Mas 09}, are coined Demazure atoms.
Azenhas, in a presentation in {\em The 69th Séminaire Lotharingien de Combinatoire} \cite{AzM 12}, identified some type $C$ Kashiwara-Nakashima tableaux as key tableaux, but does not give a construction of the right key map. Motivated by Azenhas \cite{AzM 12} and inspired by Lascoux and Schützenberger \cite{LasSchu 90}, we give a construction of left and right keys of a type $C$ Kashiwara-Nakashima tableau. Our construction, based on type $C$ frank words, introduced in Section \ref{SecKeys}, and Sheats \emph{jeu de taquin}, allows us to prove Theorem \ref{DemKey}, a type $C$ analogue of \cite[Theorem $3.8$]{LasSchu 90}. We also show, in Section \ref{SecLusztig}, that both keys are related via the Schützenberger evacuation in type $C$, or Lusztig involution, explicitly realized here using Baker-Lecouvey insertion or Sheats \emph{jeu de taquin}.
 During the preparation of the paper \cite{San 19}, Jacon and Lecouvey informed us about their paper \cite{JaLec 19}, where, with a different approach, they find the same key map in type $C$. 
%Our work has been motivated by the questions raised  in a presentation by Azenhas \cite{AzM 12}, in {\em The 69th Séminaire Lotharingien de Combinatoire}. In those questions, Azenhas identified some Kashiwara-Nakashima tableaux as key tableaux, which match our identification, but it lacks a construction of the right key map, thus lacking a proof of the combinatorial description of type $C$ Demazure characters. During the preparation of this paper, Jacon and Lecouvey informed us about their paper \cite{JaLec 19}, where they find the same key in type $C$, but their  approach is different from ours.
%Inspired by the Lascoux-Schützenberger's construction of the left and right keys  of a semistandard Young tableau \cite{LasSchu 90}, we give a similar construction in type $C$. Our construction of the left and right keys of a Kashiwara-Nakashima tableau, in type $C$, is based on frank words in type $C$, that we introduce in Section \ref{SecKeys}, and Sheats symplectic \emph{jeu de taquin}. Our Theorem \ref{DemKey} is the type $C$ analogue of \cite[Theorem $3.8$]{LasSchu 90}. We also show, in Section \ref{SecLusztig}, that both keys are related via the Schützenberger involution in type $C$, or Lusztig involution, realized here in an explicit way, using symplectic insertion or sliding operations. 
In the model of alcove paths, Lenart defined an initial key and a final key \cite{Len 07}, for any Lie type, related via the Lusztig involution. % which, in type C, have a similar behaviour to the left and right keys defined here.
%That final key can be computed as the last direction for paths in the alcove path model \cite{Len 07}.
%Since right and left keys in type $C$ are explicitly related through the Sch\"{u}tzenberger involution in type $C$, or Lusztig involution, 
There is a crystal isomorphism between the alcove path model and the Kashiwara-Nakashima tableau  model in types A and C \cite{LenLub 15}. The key maps in types $A$ and $C$ coincide in the alcove and tableau models.

The paper is organized as follows. In Section \ref{SecBruhat}, we discuss the Weyl group of type $C$, $B_n$, the Bruhat order on $B_n$ and on its cosets modulo the stabilizer of $\lambda$, the Kashiwara-Nakashima tableaux and the symplectic key tableaux. Those key tableaux are used in Proposition \ref{minimalBn} to explicitly construct the minimal length coset representatives and, 
recalling some results from Proctor \cite{Pro 81}, Theorem \ref{keycoset} gives a tableau criterion for the Bruhat order on $B_n$ and on those cosets. 
Section \ref{SecRS} briefly recalls Baker-Lecouvey insertion, the Sheats \textit{jeu de taquin} and Robinson-Schensted type $C$ correspondence, to discuss the plactic and coplactic monoids of type $C$. These monoids describe connected components and crystal isomorphic connected components of type $C$ Kashiwara crystal, for a $U_q(sp_{2n})$-module.
In Section \ref{SecKeys}, we extend the concept of frank word, in type $A$, to type $C$, and %, with the help of symplectic \textit{jeu de taquin}, 
 our Theorem \ref{rightkeymap} gives right and left key maps. Using the right key map, Theorem \ref{DemKey}, our main result, describes the tableaux that contribute to a Demazure crystal atom and to a Demazure crystal in type $C$. In Section \ref{SecLusztig}, we develop a type $C$ evacuation within the plactic monoid, an analogue of the $J$-operation discussed by Schützenberger for semistandard Young tableaux in \cite{Schu 76}. Theorem \ref{key&Lusz} shows that the evacuation of the right key of a Kashiwara-Nakashima tableau is the left key of the evacuation of the same tableau. %This is an explicit realization of Lusztig involution using insertion and sliding operations in type $C$.
 
 \textbf{Caution:} Operators, maps and group actions act on the right.

\section{Weyl group of type $C$, Bruhat order and symplectic key tableau}\label{SecBruhat}

Fix $n\in \N_{>0}$. Define the sets $[n]=\{1< \dots< n\}$ and $[\pm n]=\{1<\dots< n< \overline{n}< \dots< \overline{1}\}$ where $\overline{i}$ is just another way of writing $-i$.
The hyperoctahedral group is the group, $B_n$, with generators $s_i,\, 1\leq i\leq n$, subject to the relations:
$s_i^2 = 1, \, 1 \le i \le n;\; (s_is_{i+1})^3 = 1, 1 \le i \le n - 2;\;(s_{n-1}s_n)^4=1;(s_is_j)^2 = 1, \,1 \le i < j \le n,\, |i - j| > 1$. This group is a Coxeter group and we consider the (strong) Bruhat order on its elements \cite{BB05}. Theorem \ref{keycoset} gives a symplectic tableau criterion for this order in $B_n$.
The elements of $B_n$ can be seen as odd bijective maps from $[\pm n]$ to itself. %, i.e., for all $\sigma\in B_n$ we have $\sigma(i)=-\sigma(-i),\, i\in [\pm n]$. 
The subgroup with the generators $s_1, \dots, s_{n-1}$ is the symmetric group $\mathfrak{S}_n$. % and its elements can be seen as bijections from $[n]$ to itself. Both groups can also be seen as groups of $n\times n$ matrices. The elements of the symmetric group can be identified with the permutation matrices, and if we allow the non-zero entries to be either $1$ or $-1$, we have the elements of $B_n$. Hence $B_n$ has $2^nn!$ elements. 
The groups $\mathfrak{S}_n$ and $B_n$ are the Weyl groups for the root systems of types $A_{n-1}$ and $C_n$, respectively.
Given $\sigma\in B_n$, $\sigma=[a_1\, a_2\,\dots \,a_n]$, where $a_i=(i)\sigma$ for $i \in [n]$, is the window notation of $\sigma$. Given a vector $v\in \Z^n$,  $s_i$, with $i\in [n]$, acts on $v$, $vs_i$, swapping the $i$-th and the $(i+1)$-th entries, if $i\in [n-1]$, or changing the sign of the last entry, if $i=n$. %The window notation of $s_i\sigma $ is obtained after applying $s_i$ to the window notation of $\sigma$, if we see it as a vector. 
%Ignoring signs, $v \sigma =(v_{\sigma^{-1}(1)}, \dots, v_{\sigma^{-1}(n)})$, with $v=(v_1, \dots, v_n)$. The $i$-th letter of $v \sigma$ changes its sign if and only if $\overline{i}$ appears in $\sigma$. 
%We have that $v \sigma =(sgn(\sigma^{-1}(1)) v_{\left|\sigma^{-1}(1)\right|}, \dots, sgn(\sigma^{-1}(n))v_{\left|\sigma^{-1}(n)\right|})$, where $sgn(x)=1$ if $x$ is positive and $-1$ if $x$ is negative, for $x\in [\pm n]$.
%\begin{example}
%	Consider $v=(1,2,3)\in \Z^3$ and $\sigma=[2\,\overline{3}\,1]\in B_3$. So $\sigma=s_2s_3s_1$ and  \begin{align*}
%	(1,2,3)\sigma&=(1,2,3)s_2s_3s_1=(1,3,2)s_3s_1=(1,3,\overline{2})s_1=(3,1,\overline{2})\\
%	&=(sgn(\sigma^{-1}(1)) v_{\left|\sigma^{-1}(1)\right|}, sgn(\sigma^{-1}(2))v_{\left|\sigma^{-1}(2)\right|}, sgn(\sigma^{-1}(3))v_{\left|\sigma^{-1}(3)\right|})\\&=(1\times 3, 1\times 1, -1\times 2).
%	\end{align*}
%\end{example}
%\subsection{Bruhat order on $B_n$}
The length of $\sigma\in B_n$, $(\sigma)\ell$, is the least number of generators of $B_n$ needed to go from $[1\,2\,\dots\,n]$, the identity map, to $\sigma$. Any expression of $\sigma$ as a product of $(\sigma)\ell$ generators of $B_n$ is called reduced. %We say that two letters of the window notation of $\sigma$ form an inversion if the bigger letter appears first, and
\subsection{Kashiwara-Nakashima tableau in type $C$}
We recall the symplectic tableaux used by
Kashiwara and Nakashima to label the vertices of the type $C$ crystal graphs \cite{NK 94}, which are a variation of De Concini tableaux \cite{DeC 79}.
%\begin{definition}
	A vector $\lambda=(\lambda_1,\dots, \lambda_n)\in\Z^n$ is a partition of $|\lambda|=\sum\limits_{i=1}^n \lambda_i$ if $\lambda_1\geq \lambda_2\geq\dots\geq\lambda_n\geq 0$.
%\end{definition}
A partition $\lambda$ is identified with its Young diagram of shape $\lambda$, an array of boxes, left justified, in which the $i$-th row, from top to bottom, has $\lambda_i$ boxes.
%
%The Young diagram of shape $\lambda$ is an array of boxes, left justified, in which the $i$-th row,  We identify a partition with its Young diagram.
%\begin{example}
	For example, the Young diagram of shape $\lambda=(2,2,1)$ is $\YT{0.14in}{}{
		{{},{}},
		{{},{}},
		{{}}}$.
%\end{example}
Given $\mu$ and $\nu$ two partitions with $\nu\leq \mu$ entrywise, we write $\nu\subseteq \mu$. The Young diagram of shape $\mu/\nu$ is obtained after removing the boxes of the Young diagram of $\nu$ from the Young diagram of $\mu$.
%\begin{example}
%	For example, the Young diagram of shape $\mu/\nu=(2,2,1)/(1,0,0)$ is $\begin{tikzpicture}[scale=.4, baseline={([yshift=-.8ex]current bounding box.center)}]
%			\draw (1,0) rectangle +(-1,-1);
%			\draw (1,-1) rectangle +(-1,-1);
%			\draw (0,-2) rectangle +(-1,-1);
%			\draw (0,-1) rectangle +(-1,-1);
%	\end{tikzpicture}$\,.
%\end{example}
%\begin{definition}
%	Let $\nu\subseteq \mu$ be two partitions and $A$ a completely ordered alphabet. 
A semistandard Young skew tableau of shape $\mu/\nu$ on the completely ordered alphabet $A$ is a filling of the diagram $\mu/\nu$ with letters from $A$, such that the entries are strictly increasing in each column and weakly increasing in each row. When $|\nu|=0$ we obtain a %rectified semistandard Young skew tableau, or 
semistandard Young tableau (SSYT) of shape $\mu$.
%\end{definition}
Denote by $SSYT(\mu/\nu, A)$
the set of all semistandard skew Young tableaux $T$ of shape $\mu/\nu$, with entries in $A$. When $A=[n]$, we write $SSYT(\mu/\nu, n)$.

%When considering tableaux with entries in $[\pm n]$, it is usual to have some extra conditions besides being semistandard. We will use a family of tableaux known as \emph{Kashiwara-Nakashima} tableaux.
From now on we consider tableaux on the alphabet $[\pm n]$.
A \emph{column} is a strictly increasing sequence of numbers in $[\pm n]$ and it is usually displayed vertically.
A column is said to be \emph{admissible} if the following column condition (1CC) holds for that column:

\begin{definition}[1CC]\label{1CC}
	Let $C$ be a column. The $1CC$ holds for $C$ if for all pairs $i$ and $\overline{i}$ in $C$, where $i$ is in the $a$-th row counting from the top of the column, and $\overline{i}$ in the $b$-th row counting from the bottom, we have $a+b\leq i$.
\end{definition}
If a column $C$ is admissible then $C$ has at most $n$ letters.
If not, we say that $C$ is not admissible at $z$, where $z$ is the minimal positive integer such that $z$ and $\overline{z}$ exist in $C$ and there are more than $z$ numbers in $C$ with absolute value less or equal than $z$. 
	For instance, the column $\YT{0.15in}{}{
		{{1}},
		{{2}},
		{{\overline{1}}}}$ is not admissible at $1$.
We now define splittable columns:
%AAAAAAAAAA Refs
%The following definition states conditions to when $C$ can be split:
\begin{definition}
	Let $C$ be a column and let $I = \{z_1 > \dots > z_r\}$ be the
	set of unbarred letters $z$ such that the pair $(z, \overline{z})$ occurs in $C$. The column
	$C$ can be split when there exists a set of $r$ unbarred
	letters $J = \{t_1 > \dots > t_r\} \subseteq [n]$ such that $t_1$ is the greatest letter of $[n]$ satisfying $t_1 < z_1$,  $t_1 \not\in C$, and $\overline{t_1}\not\in C$; and for $i=2, \dots, r$, $t_i$ is the greatest letter of $[n]$ satisfying $t_i < \min(t_{i-1},   z_i)$,  $t_i \not\in C$, and $\overline{t_i} \not\in C$.\end{definition}
	A column $C$ is admissible if and only if $C$ can be split  \cite[Lemma 3.1]{She 99}.
	If $C$ can be split then we define right column of $C$, $Cr$, and the left column of $C$, $ C\ell$. The column $Cr$ is obtained by replacing, in $C$,  $\overline{z_i}$ with $\overline{t_i}$ for each letter $z_i \in I $ and reordering, if needed; $ C\ell$ is obtained after replacing $z_i$ with $t_i$ for each letter $z_i \in I $ and reordering, if needed.
	If $C$ is admissible then $ C\ell\leq C \leq Cr$ by entrywise comparison. If $C$ does not have symmetric entries, then $C$ is admissible and  $ C\ell= C =Cr$.
%In the next definition we give conditions for a column $C$ to be coadmissible.
%	\begin{definition}
%		We say that a column $C$ is coadmissible if for every pair $i$ and $\overline{i}$ on $C$, where $i$ is on the $a$-th row counting from the top of the column, and $\overline{i}$ on the $b$-th row counting from the top, then $b-a\leq n-i$.
%	\end{definition}
%	%Note that, unlike in Definition \ref{1CC}, in the last definition $b$ is counted from the top of the column.
%	
%	
%	Given an admissible column $C$, consider the function $\Phi$ that sends $C$ to the column of the same size in which the unbarred entries are taken from $\ell C$ and the barred entries are taken from $rC$. The column $\Phi(C)$ is coadmissible and the algorithm to form $\Phi(C)$ from $C$ is reversible \cite[Section 2.2]{Lec 02}. %In particular, every column on the alphabet $[n]$ is simultaneously admissible and coadmissible. 
%	
%\begin{example} Let
%	$C=\YT{0.17in}{}{
%		{{2}},
%		{{3}},
%		{{\overline{3}}}}$ be an admissible column. Then $\ell C=\YT{0.17in}{}{
%		{{1}},
%		{{2}},
%		{{\overline{3}}}}$ and $rC=\YT{0.17in}{}{
%		{{2}},
%		{{3}},
%		{{\overline{1}}}}$. So $\Phi(C)=\YT{0.17in}{}{
%		{{1}},
%		{{2}},
%		{{\overline{1}}}}$ is coadmissible.
%\end{example}
%	
	Let $T$ be a skew tableau with all of its columns admissible. The split form of a skew tableau $T$, $(T)spl$, is the skew tableau obtained after replacing each column $C$ of $T$ by the two columns $ C\ell\,Cr$. The tableau $(T)spl$ has double the amount of columns of $T$.
%	\begin{definition}
		A semistandard skew tableau $T$ is a Kashiwara-Nakashima (KN) skew tableau if its split form is a semistandard skew tableau. We define $\mathcal{KN}(\mu/\nu, n)$ to be the set of all KN tableaux of shape $\mu/\nu$ in the alphabet $[\pm n]$. When $|\nu|=0$ we obtain $\mathcal{KN}(\mu, n)$. 
%	\end{definition}
When $T\in SSYT(\mu/\nu, [\pm n])$ with no symmetric entries in any of its columns, $T$ is a KN skew tableau. In particular $SSYT(\mu/\nu, n)\subseteq \mathcal{KN}(\mu/\nu, n)$.

The weight of a word $w$, $(w)\text{wt}$, on the alphabet $[\pm n]$ is the vector in $\Z^n$ where the entry $i$ is the multiplicity of the letter $i$ minus the multiplicity of the letter $\overline{i}$, for $i\in [n]$. The length of $w$ is its number of letters.
The column reading word of a KN tableau $T$, $(T)cr$, is obtained reading down columns, right to left. %the word read in $T$ in the Japanese way.
 The weight of $T$ is the vector $(T)\text{wt}:=((T)cr)\text{wt}$. %=(t_1-t_{\overline{1}}, t_2-t_{\overline{2}},\dots, t_n-t_{\overline{n}})\in \Z^n$, where $t_\alpha$ is the number of $\alpha$'s in $T$, with $\alpha \in [\pm n]$.
%\begin{example}
Let
	$T=\YT{0.15in}{}{
		{{2},{2}},
		{{3},{3}},
		{{\overline{3}}}}$ and $n=3$. The split form of
	$T$ is the tableau $(T)spl=\YT{0.15in}{}{
		{{1}, {2},{2},{2}},
		{{2},{3},{3},{3}},
		{{\overline{3}},{{\overline{1}}}}}$. Hence $T\in \mathcal{KN}((2,2,1),3)$.
  Also $(T)cr=23\,23\overline{3}$ and $(T)\text{wt}=((T)cr)\text{wt}=(0,2,1)$.%\end{example}
%In Section \ref{BakerIns}, we recall a way to go from a word in the alphabet $[\pm n]$ to a KN tableau, the Baker-Lecouvey insertion.
%\subsection{Key tableaux in type $C$}\label{keyBruhat}

Given a partition $\lambda\in\Z^n$, the $B_n$-orbit of $\lambda$ is the set $\lambda B_n:=\{\lambda\sigma\mid \sigma\in B_n\}$.
\begin{definition}
	A key tableau in type $C$, on the alphabet $[\pm n]$, is a KN tableau in $\mathcal{KN}(\lambda,n)$, for some partition $\lambda$, in which the set of elements of each column is contained in the set of elements of the previous column and the letters $i$ and $\overline{i}$ do not appear simultaneously as entries, for any $i\in [n]$.
%\begin{example}
See the example at the end of Section \ref{Ssec Bruhat}.
%\end{example}
\end{definition}
%
%The set of key tableaux in type $A$ is the subset of the key tableaux in type $C$ consisting of the tableaux having only positive entries, hence they are SSYT for the alphabet $[n]$.
%
%Every vector $v$ of  $\Z^n$ is in the $B_n$-orbit of exactly one partition.  %, $\lambda_v$, which is the partition
%the one obtained  by sorting the absolute values of all entries of $v$. 
%\begin{example}
%\end{example}
%Consider a partition $\lambda\in\Z^n$. The $B_n$-orbit of $\lambda$ is the set $\lambda B_n:=\{\lambda\sigma\mid \sigma\in B_n\}$. Every vector of $\Z_n$ is in the orbit of exactly one partition,  $\lambda_v$, which is the partition obtained after ordering the absolute value of all entries of $v$.
%\begin{example}
%	Let $v=(1,\overline{3},0, 3, \overline{2})$. Then $\lambda_v=(3,3,2,1,0)$.
%\end{example}
Given $v=(v_1, \dots, v_n)\in \Z^n$, put in the first $|v_i|$ columns the letter $i$ if $v_i>0$ or $\overline{i}$ if $v_i<0$. This defines a key tableau of weight $v$, $(v)K$.
\begin{proposition}\label{UniKw}
	Let $v\in \lambda B_n$. There is exactly one key tableau $(v)K$ whose weight is $v$. The shape of $(v)K$ is $\lambda$. $(\lambda)K$ is the only KN tableau of weight and shape $\lambda$. % also called Yamanouchi tableau of shape $\lambda$. 
	The map $v\mapsto (v)K$ is a bijection between $\lambda B_n$ and key tableaux in $\mathcal{KN}(\lambda, n)$.
\end{proposition}

\subsection{The Bruhat order on $B_n$ and cosets of $B_n$}\label{Ssec Bruhat}
Given a partition $\lambda\in \Z^n$, let $W_\lambda=\{\rho\in B_n\mid \lambda\rho=\lambda\}$ be the stabilizer of $\lambda$, under the action of $B_n$,  a standard parabolic subgroup of $B_n$ generated by a subset of simple generators. %Let $J\subseteq [n]$ be the set of the indices of the generators of $W_\lambda$, i.e. $W_\lambda=\langle s_j, j\in J \rangle$, and $J^c$ the complement of this set in $[n]$.
Let $W_\lambda\setminus B_n=\{W_\lambda\sigma:\sigma\in B_n\}$ be the set of right cosets of $B_n$ determined by the subgroup $W_\lambda$. Given a right coset in $W_\lambda\setminus B_n$, all its elements return the same vector when acting on $\lambda$. Hence the vectors $v$ in the $B_n$-orbit of $\lambda$ define a labelling for the right cosets. 
Therefore, the symplectic key tableaux in $\mathcal{KN}(\lambda, n)$ and the cosets of $B_n$, modulo $W_\lambda$, are in bijection: $(v)K\leftrightarrow v \leftrightarrow W_\lambda\sigma_v$, where $\sigma_v$ is the minimal length coset representative.
%
% Each coset $W_\lambda\sigma$ has a unique minimal length element $\sigma_v$ such that $v=\lambda\sigma_v$. Reciprocally, given a vector $v\in \lambda B_n$,  there is a unique minimal length element $\in B_n$ such that $v=\lambda\sigma_v$. We have then a bijection between key tableaux of shape $\lambda$ on the alphabet $[\pm n]$,
% vectors in the $B_n$-orbit of $\lambda$, and right cosets of $B_n$ determined by the subgroup $W_\lambda$, %given by %\begin{align*}$\lambda B_n\rightarrow B_n/W_\lambda$ 
% via $K(v)\leftrightarrow v\leftrightarrow W_\lambda \sigma_v$.
%\end{align*}
%The set $J^c$ detects the minimal length coset representatives of $W_\lambda\setminus B_n$: $\sigma$ is a minimal coset representatives of $W_\lambda\setminus B_n$ if and only if all its reduced decompositions starts with a generator $s_i\in J^c$ \cite{BB05}.
Key tableaux, $(v)K$, $v\in\lambda B_n$, may be used
to explicitly construct the minimal length coset representatives of  $W_\lambda\setminus B_n$, a generalization of what Lascoux does for vectors in $\mathbb{N}^n$ (hence $\sigma_v \in \mathfrak{S}_n$).
%Given a vector $v\in \lambda B_n$, we show that there is a unique minimal length element $\sigma_v\in B_n$ such that $v=\lambda\sigma_v$ and we show how to obtain $\sigma_v$ explicitly. %The next proposition is a generalization of what Lascoux does in \cite{Las 03} for vectors in $\N^n$ .
%We first construct the minimal length coset representative of a $B_n$-orbit.
%In the next proposition, given a vector $v\in \lambda B_n$, we show that there is a unique minimal element $\sigma_v\in B_n$ such that $v=\lambda\sigma_v$ and we show how to obtain $\sigma_v$ explicitly. The proposition is a generalization of what Lascoux does in \cite{Las 03} for vectors in $\N^n$ (hence $\sigma_v \in \mathfrak{S}_n$).
%\vspace{-.5cm}
\begin{proposition}\label{minimalBn}
	Let $v\in \lambda B_n$ and $T$ the tableau obtained after adding the column $\YT{0.15in}{}{
		{{1}},
		{{2}},
		{{\vdots}},
		{{n}}}$ to the left of $(v)K$. The aforementioned minimal length coset representative $\sigma_v$ is given by the reading word $T$, where entries with the same absolute value are read just once. 
\end{proposition}
Let	$v=(3,\overline{3},0, 0, \overline{2})$.
Then
$(v)K=\YT{0.15in}{}{
	{{1},{1},{1}},
	{{\overline{5}},{\overline{5}},{\overline{2}}},
	{{\overline{2}},{\overline{2}}}}$, $T=\YT{0.15in}{}{
	{{1}, {1},{1},{1}},
	{{2},{\overline{5}},{\overline{5}},{\overline{2}}},
	{{3},{\overline{2}},{\overline{2}}},
	{{4}},{{5}}}
$ and $\sigma_v=[1\overline{2}\overline{5}34]$.

Given  $v$ and $u$  in $\lambda B_n$, we write $v\le u$ to mean $\sigma_v\leq\sigma_u$ in the Bruhat order. Put $\Lambda_n=(n, n-1, \dots, 1)$.
Thanks to Theorem $3BC$ of Proctor's Ph.D. thesis \cite{Pro 81}, we have a tableau criterion for the Bruhat order on vectors in the same $B_n$-orbit.% and for the corresponding $B_n$-coset space.
\begin{theorem}\cite[Theorem $3BC$]{Pro 81}\label{keycoset}
	Let $v,u\in\lambda B_n$. Then $\sigma_v\leq\sigma_u$ if and only if $(v)K\leq (u)K$, by entrywise comparison. In particular, for $\sigma, \rho \in B_n$, $\sigma\leq \rho \Leftrightarrow (\Lambda_n \sigma)K\leq (\Lambda_n \rho)K$.
\end{theorem}
\noindent
For instance,	$v\!=\!(3,\overline{3},0, 0, \overline{2})\!\leq\! u\!=\!(\overline{3},2,0,\overline{3},0)$, because 
$(v)K\!\!=\!\!\YT{0.15in}{}{
	{{1},{1},{1}},
	{{\overline{5}},{\overline{5}},{\overline{2}}},
	{{\overline{2}},{\overline{2}}}}$ $
\le$ $ (u)K\!\!=\!\!\YT{0.15in}{}{
	{{2},{2},{\overline{4}}},
	{{\overline{4}},{\overline{4}},{\overline{1}}},
	{{\overline{1}},{\overline{1}}}}
$.% and $\sigma_v=[1\overline{2}\overline{5}34]\le \sigma_u=[\overline{4}\overline{1}235]$.
%\begin{proof}
%Let $v,\,u\in \lambda B_n$ two vectors in the same $B_n-orbit$.We have that 
%\begin{align*}
%\sigma_v\leq\sigma_u \stackrel{(1)}{\Leftrightarrow} v \leq u\stackrel{(2)}{\Leftrightarrow} \widetilde{v} \leq \widetilde{u} \stackrel{(3)}{\Leftrightarrow} K(\widetilde{v}) \leq K(\widetilde{u})\Leftrightarrow \widetilde{K(v)} \leq \widetilde{K(u)}\stackrel{(4)}{\Leftrightarrow} K(v) \leq K(u),
%\end{align*}
%where $(1)$ holds by Definition \ref{BruhatVec}. Note that in $(2)$ we also need to record $\lambda$, because it is needed in $(4)$ to recover the shape of $K(v)$ from the shape $\widetilde{K(v)}$. Finally the equivalence $(3)$ is an application of Theorem 3BC of Proctor's Ph.D. thesis \cite{Pro 81}.
%\end{proof}
%One has a natural bijection between $B_n$ and the $B_n$-orbit of $\Lambda=(n, n-1, \dots, 1)$. Hence, we a tableau criterion for the Bruhar order of $B_n$ given by $\sigma\leq \rho \Leftrightarrow K(\Lambda \sigma)\leq K(\Lambda \rho)$.

%\section{Crystal graphs in type $C$ and symplectic plactic monoid}\label{SecRS}
\section{Type $C$ crystal graphs, plactic and coplactic monoids}\label{SecRS}
%We recall two equivalence relations of words in the alphabet $\left[\pm n\right]$, the Knuth type $C$ equivalence, or (symplectic) plactic equivalence, and the (symplectic) coplactic equivalence. On the basis of these two equivalence relations is 
Let $[\pm n]^\ast$ be the free monoid on the alphabet $[\pm n]$. 
Recall the type $C_n$ simple roots $\{\alpha_i=\mathbf{e_1}-\mathbf{e_2}, i\in [n-1]\}\cup \{\alpha_n=2\mathbf{e_n}\}$. Here a Kashiwara crystal of type $C_n$ is a nonempty set $\mathfrak{B}$ together with the following maps and statistics \cite{BSch 17}:
$e_i, f_i:\mathfrak{B}\rightarrow \mathfrak{B}\sqcup\{0\}$, 
$ \varepsilon_i, \varphi_i:\mathfrak{B}\rightarrow \Z$,
$\text{wt}: \mathfrak{B}\rightarrow \Z^n$,
%\end{align*}
where $i\in [n]$ and $0\notin \mathfrak{B}$ is an auxiliary element, such that:
%\begin{enumerate}
if $a, b \in \mathfrak{B}$ then $(a)e_i=b\Leftrightarrow (b)f_i=a$ and in this case $(b)\text{wt}=(a)\text{wt}+\alpha_i$, $(b)\varepsilon_i=(a)\varepsilon_i-1$ and $(b)\varphi_i=(a)\varphi_i+1$;
for all $a \in \mathfrak{B}$, we have $(a)\varphi_i=\langle (a)\text{wt}, \frac{2\alpha_i}{\langle \alpha_i, \alpha_i\rangle}\rangle+(a)\varepsilon_i$, where $\langle ,\rangle$ is the usual inner product in $\R^n$.
For all $a \in \mathfrak{B}$, we have %$\varphi_i(a)=\langle \text{wt}(a), \frac{2\alpha_i}{\langle \alpha_i, %\alpha_i\rangle}\rangle+\varepsilon_i(a)$.
%\end{enumerate}
%%Fun\c{c}\~{a}o peso
%The crystals we deal with are the ones of a $U_q(sp_{2n})$-module. They are seminormal \cite{BSch 17}, and satisfy
$(a)\varphi_i=\max\{k\in \Z_{\geq 0}\mid (a)f_i^k\neq 0\}$ and  $(a)\varepsilon_i=\max\{k\in \Z_{\geq 0}\mid (a)e_i^k\neq 0\}$.
%and $wt(a)=\sum_{i\in I}(\varphi_i(a)-\varepsilonon_i(a))\overline{\omega_i}$.
An element $u\in \mathfrak{B}$ such that $(u)e_i=0$ (or $(u)f_i=0$) for all $i\in [n] $ is called  a \emph{highest weight element} (or \emph{lowest weight element}).
We associate with $\mathfrak{B}$ a coloured oriented graph with weighted vertices in $\mathfrak{B}$ and edges labelled by $i\in [n]$: $b\overset{i}{\rightarrow} b'$ if and only if $b'=(b)f_i$, $i\in [n]$, $b, b'\in\mathfrak{B}$. This is the \emph{crystal graph} of $\mathfrak{B}$. The $C_n$ standard crystal $\mathbb{B}$ is
$1\xrightarrow{1} 2\xrightarrow{2} \dots\xrightarrow{n-1} n\xrightarrow{n} \overline{n}\xrightarrow{n-1}\dots \xrightarrow{1}1$, with set
$\mathbb{B}=[\pm n]$, 
where $(i)\text{wt}=\bf e_i$, $(\overline{i})\text{wt}=-\bf e_i$. The highest weight word is the word $1$, and the lowest weight word is $\overline{1}$.

Let $\mathfrak{B}$ and $\mathfrak{C}$ be two crystals associated to the same root system. The \emph{tensor product} $\mathfrak{B}\otimes \mathfrak{C}$ is a crystal whose set is the Cartesian product $\mathfrak{B}\times\mathfrak{C}$, where its elements are $b\otimes c$, $b\in\mathfrak{B}$ and $c\in\mathfrak{C}$, with $(b\otimes c)\text{wt}=(b)\text{wt}+(c)\text{wt}$. The crystal operator $f_i$ is defined by $(b\otimes c)f_i=\begin{cases}
(x)f_i\otimes y\; \text{if}\; (c)\varphi_i\leq(b)\varepsilon_i\\
x\otimes (y)f_i\; \text{if}\; (c)\varphi_i>(b)\varepsilon_i
\end{cases}$ $\!\!\!\!\!\!\!\!$, and $e_i$ is its inverse. Using the tensor product we can define the crystal $\mathbb{B}^{\otimes k}$ of words of length $k$. Thus, we define how the crystal operators $f_i$ and $e_i$ act on any finite word. This operators can be described via the signature rule, see \cite{BSch 17}. Let $G_n=\bigoplus\limits_{k\geq 0}\mathbb{B}^{\otimes k}$ be the type $C_n$ crystal of all words in $[\pm n]^\ast$. The  crystal $G_n$ is the union of connected components where each component has a unique highest (lowest) weight word. Two connected components are isomorphic if and only if they have the same highest weight \cite{Lec 02}.

The Robinson-Schensted (RS) type $C$ correspondence is a bijection between words $w\in G_n$ of length $k$, and tuples consisting of a KN tableau $(w)P$ % of partition shape
 and an \emph{oscillating tableau} $Q$, of length $k$, with the same final shape as $(w)P$, see \cite{Lec 02}.
%AAAAAAAAAAAAAAAAAAAAAAAAAAAAAA ir ao Lecouvey
 We denote this map by $w\mapsto ((w)P, Q)$, where $(w)P$ can be computed via Sheats \emph{jeu de taquin} or \emph{Baker-Lecouvey insertion}. 
 %The type $C$ RS correspondence has a natural interpretation in terms of crystal connectivity and crystal isomorphic connected components in Kashiwara's crystal graphs for $U_q(sp_{2n})$ \cite{BSch 17, Lec 02}. %In the last part of this section, we will study some subcrystals of these crystal graphs known as Demazure crystal and describe them using key tableau.
%For this aim and reader convenience, we begin to recall the symplectic \emph{jeu de de taquin} and Baker-Lecouvey insertion.
%\subsection{Sheats symplectic \textit{jeu de taquin} and Baker-Lecouvey insertion}
%We briefly recall the Sheats symplectic \textit{jeu de taquin} and Baker-Lecouvey insertion.
The symplectic \textit{jeu de taquin} \cite{Lec 02, She 99} is a weight-preserving procedure that allows us to change the shape of a  KN skew tableau and eventually rectify it, i.e., make it to have partition shape. It is a variation of the ordinary \textit{jeu de taquin} for skew SSYTs.
The rectification is independent of the order in which the inner corners of $\nu$  are filled \cite[Corollary 6.3.9]{Lec 02}. %Given the KN skew tableau
%	$\SYT{0.145in}{}{
%		{{2}},
%		{{3},{1}},
%		{{\overline{1}},{2}}}$, to rectify it via symplectic \textit{jeu taquin}, we start by splitting it. The first two slides are vertical, obtaining	$\SYT{0.15in}{}{
%		{{2}, {2},{1},{1}},
%		{{3},{3},{2},{2}},
%		{{\overline{1}},{\overline{1}}}}$. Now we do an horizontal slide in which we take $\overline{1}$ from the second column of $T$ and adding it to the the coadmissible column of the first column of $T$, obtaining
%	$\YT{0.145in}{}{
%		{{2},{2}},
%		{{3},{3}},
%		{{\overline{3}}}}$.
%\end{example}
%\begin{remark}
%	If the columns $C_1$ and $C_2$ do not have negative entries then the symplectic \textit{jeu de taquin} coincides with the \textit{jeu de taquin} known for SSYT.
%\end{remark}
%\subsection{Baker-Lecouvey insertion}\label{BakerIns}

The Baker-Lecouvey insertion \cite{Ba 00,Lec 02} is a bumping algorithm that, given a word $w$ in the alphabet $[\pm n]$, returns the KN tableau $(w)P$.
 It depends on the symplectic jeu de taquin.  This insertion is similar to the usual column insertion for SSYTs. 
However, when inserting a letter it may happen that we remove a cell from the inserted tableau, instead of adding. The length of $((w)P)cr$ might be less than the length of $w$, but the weight is preserved, $(w)\text{wt}=((w)P)\text{wt}$. 
If the word $w$ does not have symmetric letters, then the insertion works just like the column insertion for SSYTs.
%For instance, consider the word $w=2\overline{1}1$.
%		The Baker-Lecouvey insertion of $w$ creates the sequence of tableaux
%		$\YT{0.15in}{}{
%			{{{2}}}}
%		\YT{0.15in}{}{
%			{{2}},
%			{{\overline{1}}}}
%		\YT{0.15in}{}{
%			{{2},{2}},
%			{{\overline{2}}}}=P(2\overline{1}1)$.		
%		The SSYT column insertion of $w$ results in the tableau $\YT{0.15in}{}{
%			{{1},{2}},
%			{{\overline{1}}}}$, which is not a KN tableau because the first column is not admissible
%\begin{example}\label{ExBakIns}
%	Consider the word $w=23\overline{2}\overline{3}1$. We now insert $w$, obtaining the following sequence of tableaux:	
%	$\YT{0.15in}{}{
%		{{{2}}}}
%	\YT{0.15in}{}{
%		{{2}},
%		{{3}}}
%	\YT{0.15in}{}{
%		{{2}},
%		{{3}},
%		{{\overline{2}}}}
%	\YT{0.15in}{}{
%		{{1},{\overline{1}}},
%		{{3}},
%		{{\overline{3}}}}
%	\YT{0.15in}{}{
%		{{1},{1},{\overline{1}}},
%		{{3}},
%		{{\overline{3}}}}=P(w).$	
%	Note that the insertion of the fourth letter, $\overline{3}$, causes a type I special bump on the first column and  the insertion of the fifth letter, $1$, causes a type IIb special bump on the second column.
%\end{example}
%The symplectic jeu de taquin is reversable, which means that given a skew tableau we can "rectify" it to the bottom-right corner. 
 If $l$ is the length of $w$, $(w)P$ is the rectification of the skew tableau of shape $\Lambda_l/\Lambda_{l-1}$ and reading word $w$ \cite[Corollary 6.3.9]{Lec 02}. %\end{proposition}
%\begin{proposition}\label{rr}
 More generally, if $T\in \mathcal{KN}(\mu/\nu, n)$, the rectification of $T$ coincides with $((T)cr)P$.
Given $w_1, w_2 \in [\pm n]^\ast$, the relation
$w_1 \sim w_2\Leftrightarrow (w_1)P=(w_2)P$ defines an equivalence relation on $[\pm n]^\ast$ known as Knuth equivalence. The type $C$ plactic monoid is the quotient $[\pm n]^\ast/\sim$ where each Knuth (plactic) class is uniquely identified with a KN tableau \cite{Lec 02}. Hence two Knuth-related words have the same weight. It is also described as the quotient of $[\pm n]^\ast$ by the elementary Knuth relations; see \cite{Lec 02} for details. If $w_1 \sim w_2$ then they occur in the same place in two isomorphic connected components of $G_n$ %, that is, they are obtained from two highest words with the same weight through a same sequence of crystal operators 
\cite{Lec 02}, \emph{i.e.}, $(w_1)e_i\sim (w_2)e_i$ and $(w_1)f_i\sim(w_2)f_i$, $i \in [n]$.
  Two words $w_1,w_2\in [\pm n]^\ast$ are% crystal connected or
  coplactic equivalent if and only if they belong to the same connected component of $G_n$. %, \emph{i.e.}, one is obtained from  another by some sequence of crystal operators $f_i$ and $e_j$, $i, j \in [n]$.  
  The connected components of $G_n$ are the coplactic classes in the RS correspondence that identify words with the same oscillating tableau \cite[Proposition 5.2.1]{Lec 02}. %Also, two words $w_1,w_2\in [\pm n]^\ast$ are Knuth equivalent if and only if they occur in the same place in two isomorphic connected components of $G_n$ %, that is, they are obtained from two highest words with the same weight through a same sequence of crystal operators 
 %\cite{Lec 02}, \emph{i.e.}, if $w_1\sim w_2$ then $(w_1)e_i\sim (w_2)e_i$ and $(w_1)f_i\sim(w_2)f_i$, $i \in [n]$.
 
Choose a word $w\in [\pm n]^\ast$ where the shape of $(w)P$ is $\lambda$. If we replace every word of its coplactic class with its insertion tableau we obtain the crystal $\mathfrak{B}^\lambda$ of tableaux  $\mathcal{KN}(\lambda, n)$. The crystal $\mathfrak{B}^\lambda$ does not depend on the initial choice of word $w$ in the plactic class of $w$ \cite[Theorem 6.3.8]{Lec 02}.
A word $w$ of $G_n$ is a highest weight word if and only if the weight of all its prefixes (including itself) is a partition. In this case, %one has that 
	$(w)P=(\lambda)K$.% the Yamanouchi tableau of shape $\lambda$, the weight of $w$
\section{Right and left keys and Demazure atoms in type $C$}\label{SecKeys}
 We generalize Lascoux-Sch\"{u}tzenberger frank words, in type $A$ \cite{LasSchu 90},   to type $C$ to create  right and left key maps in type $C$. %, that send KN tableaux to key tableaux in type C.
 Our Theorem \ref{DemKey} detects  the type $C$ KN tableaux for Demazure atoms. It is the type $C$ version of  Lascoux and Sch\"{u}tzenberger \cite[Theorem 3.8]{LasSchu 90}.
%We  define type C frank words in $[\pm n]^\ast$ and use them to create the right and left key maps, that send KN tableaux to key tableaux in type C. In type $A$ they were introduced by Lascoux and Schützenberger \cite{LasSchu 90}. Our Theorem \ref{DemKey} gives a combinatorial description of type $C$ Demazure atoms. It is a type $C$ of \cite[Theorem 3.8]{LasSchu 90}, due to Lascoux and Schützenberger.
%The main result of this section is the type $C$ version of  , which, using the right key map, gives  a combinatorial description of type $C$ Demazure atoms.
%In this section we will define frank words on the alphabet $[\pm n]$ and use them to create the right and left key maps, that send tableaux to key tableaux.  The climax of this section lies on a type $C$ version of the Theorem 3.8 of \cite{LasSchu 90}, giving a combinatorial description of the Demazure Crystals, some subcrystals of the crystal $B^\lambda$, using the right key map. 
%\subsection{Frank words in type $C$}
\begin{definition}
	The word $w\in [\pm n]^\ast$ is a type $C$ frank word if the lengths of its maximal column factors form a multiset equal to the multiset formed by the lengths of the columns of the tableau $(w)P$.
\end{definition}\vspace{-.2 cm}
For instance, $(23\overline{2}\overline{3}1)P=(\overline{1}113\overline{3})P=\YT{0.14in}{}{
		{{1},{1},{\overline{1}}},
		{{3}},
		{{\overline{3}}}}$. Since $23\overline{2}\overline{3}1$ and $\overline{1}113\overline{3}$ have one column of length $3$ and two columns of length $1$, they are frank words. 

Given a frank word $w$, the number of letters of $w$ is the same as the number of cells of $(w)P$. %, hence the case $3$ of the Baker-Lecouvey insertion doesn't happen.
%
%\begin{proposition}
%	Let $w$ be frank word on the alphabet $[\pm n]$. 
This implies that all columns of $w$ are admissible.
%\end{proposition} 
%\begin{proof}
%	Suppose that the statement is false. So there is a factor of $w$ that is a non-admissible column with all of its proper factors admissible. Hence we can apply the Knuth relation $K5$, meaning that $w$ is Knuth related to a smaller word $w'$. But in this case, the number of letters of $w'$ is less then the number of cells of $P(w)=P(w')$, which is a contradiction.
%\end{proof}
The following proposition is an extension of \cite[Proposition 7]{Ful 96} on SSYTs to KN tableaux.
\begin{proposition}\label{Fultonices}
	Let $T\in\mathcal{KN}(\lambda, n)$. Let $\mu/\nu$ be a skew diagram with same number of columns of each length as $T$. Then there is a unique KN skew tableau $S$ with shape $\mu/\nu$ that rectifies to $T$ and $(S)cr$ is a frank word. 
\end{proposition}
%\begin{proof}
%	If $T$ is a Yamanouchi tableau $K(\lambda)$ and $S\in \mathcal{KN}(\mu/\nu, n)$ rectifies to $K(\lambda)$,  then, since $S$ and $K(\lambda)$ have the same number of cells, all entries of $S$ are unbarred, hence $S$ is a semistandard skew tableau.  So, it follows from \cite[Proposition 7]{Ful 96} that $S$ exists and is unique. If $T$ is not a Yamanouchi tableau, note that $T$ is crystal connected to $K(\lambda)$ and from \cite[Theorem 6.3.8]{Lec 02} we have that the symplectic \textit{jeu de taquin} slides commutes with the action of the crystal operators. Consider $Y_\lambda'$ the only tableau on the skew-shape $\mu/\nu$ that rectifies to $Y_\lambda$, which exists due to \cite[Proposition 7]{Ful 96}. Since $S$ rectifies to $T$, which is crystal connected to $K(\lambda)$, and  $Y_\lambda'$ rectifies to $K(\lambda)$, $S$ is crystal connected to $Y_\lambda'$ and the path has same sequence of colours as the one from $T$ to $K(\lambda)$. Hence $S$ exists and is uniquely defined.
%\end{proof}

\begin{corollary}\label{lastcolumn} Let $S$ be as in the previous proposition. 
	The last column of $S$ depends only on the length of that column.
\end{corollary}
%\begin{proof}
%	All other skew tableaux with given last column length can be found from a given one by playing the symplectic \textit{jeu de taquin} or its reverse in all columns except the last one. Note that $S$ has the same number of cells of the tableau obtained after rectifying, hence we can't lose cells when applying the symplectic \textit{jeu de taquin} or its reverse.
%\end{proof}

Fixed a KN tableau $T$, consider the set of all possible last columns taken from skew tableaux with same number of columns of each length as $T$. Corollary \ref{lastcolumn} implies that this set has one element for each distinct column length of $T$. For each column $C$ in this set, consider the column $Cr$, its right column. The next proposition implies that this set of right columns is nested, if we see each column as the set formed by its letters.

\begin{proposition}\label{nestedcolumns}
	Consider $T$ a two-column KN skew tableau $C_1C_2$ with empty cells in the first column. Slide via symplectic \textit{jeu de taquin} the bottommost of those empty cell, obtaining a two-column KN skew tableau $C_1'C_2'$. Then $C_2'r\subseteq C_2r$.
\end{proposition}
Next, one gives the type $C$ right key map. It extends the one defined for type $A$ in \cite{LasSchu 90}.
\begin{theorem}[Right key map]\label{rightkeymap}
	Given a KN tableau $T$, if we replace each column with a column of the same size taken from the right columns of the last columns of all skew tableaux associated to $T$, then we obtain a key tableau. This tableau is the right key tableau of $T$ and we denote it by $(T)K_+$. (See Example \ref{ExRightkey}.)
\end{theorem}
%\begin{proof}
%	The previous proposition implies that the columns of $K_+(T)$ are nested and do not have symmetric entries. So, it is indeed a KN key tableau.
%\end{proof}
\begin{remark} Recall the set up of Proposition \ref{Fultonices}. If the shape of $S$, $\mu/\nu$, is such that every two consecutive columns have at least one cell in the same row, then each column of $S$ is a maximal column factor of the word $(S)cr$, hence $(S)cr$ is a frank word. Moreover, the columns of $S$ appear in reverse order in $(S)cr$. Therefore, given a KN tableau $T$, the columns of $(T)K_+$ consist of right columns of the first columns of the frank words associated to $T$.
\end{remark}
%In the same spirit of the right key, we define the left key of a KN tableau. %Given a tableau $T$, we can replace each column with a column taken from the left columns of the first columns of all skew-tableaux with the same number of columns of each length as $T$. 
In the set up of Proposition \ref{nestedcolumns}, we also can prove that $ C_1\ell\subseteq  C_1'\ell$, hence the set of left columns of the first columns of all skew tableaux with the same number of columns of each length as $T$ will be nested.
The left key $(T)K_-$ is obtained after replacing each column of $T$ with a column of the same size taken from this set.
\begin{example}\label{ExRightkey}
	The tableau $T=\YT{0.15in}{}{
				{{1},{3},{\overline{1}}},
				{{3},{\overline{3}}},
				{{\overline{3}}}}$ has the following six KN skew tableaux with same number of columns of each length as $T$, each one corresponding to a permutation of its column lengths, and each one is associated to the frank word given by its column reading. 
			
			
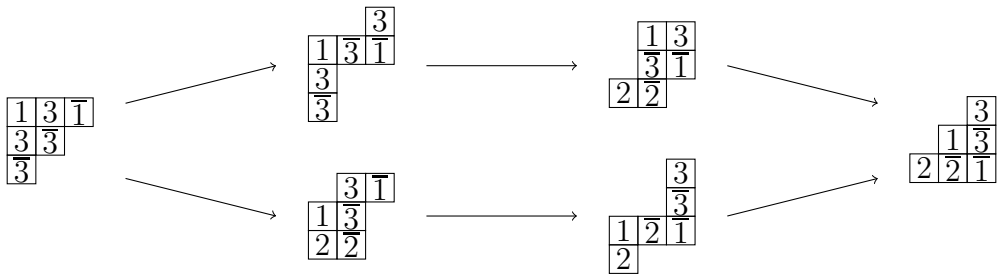
\begin{tikzpicture}
			\node at (0,0) {$\YT{0.15in}{}{
					{{1},{3},{\overline{1}}},
					{{3},{\overline{3}}},
					{{\overline{3}}}}$};	
			\node at (4,1) {$\begin{tikzpicture}[scale=.38, baseline={([yshift=-.8ex]current bounding box.center)}]
				\draw (0,0) rectangle +(1,1);
				\draw (0,1) rectangle +(1,1);
				\draw (0,2) rectangle +(1,1);
				\draw (1,2) rectangle +(1,1);
				\draw (2,2) rectangle +(1,1);
				\draw (2,3) rectangle +(1,1);
				\node at (.5,.5) {$\overline{3}$};
				\node at (.5,1.5) {$3$};
				\node at (.5,2.5) {$1$};
				\node at (1.5,2.5) {$\overline{3}$};
				\node at (2.5,2.5) {$\overline{1}$};
				\node at (2.5,3.5) {$3$};
				\end{tikzpicture}$};
			\node at (8,1) {$\begin{tikzpicture}[scale=.38, baseline={([yshift=-.8ex]current bounding box.center)}]
				\draw (0,0) rectangle +(1,1);
				\draw (1,0) rectangle +(1,1);
				\draw (1,1) rectangle +(1,1);
				\draw (1,2) rectangle +(1,1);
				\draw (2,1) rectangle +(1,1);
				\draw (2,2) rectangle +(1,1);
				\node at (.5,.5) {$2$};
				\node at (1.5,.5) {$\overline{2}$};
				\node at (1.5,1.5) {$\overline{3}$};
				\node at (1.5,2.5) {$1$};
				\node at (2.5,1.5) {$\overline{1}$};
				\node at (2.5,2.5) {$3$};
				\end{tikzpicture}$};
			\node at (4,-1) {$\begin{tikzpicture}[scale=.38, baseline={([yshift=-.8ex]current bounding box.center)}]
				\draw (0,0) rectangle +(1,1);
				\draw (0,1) rectangle +(1,1);
				\draw (1,0) rectangle +(1,1);
				\draw (1,1) rectangle +(1,1);
				\draw (1,2) rectangle +(1,1);
				\draw (2,2) rectangle +(1,1);
				\node at (.5,.5) {$2$};
				\node at (.5,1.5) {$1$};
				\node at (1.5,.5) {$\overline{2}$};
				\node at (1.5,1.5) {$\overline{3}$};
				\node at (1.5,2.5) {$3$};
				\node at (2.5,2.5) {$\overline{1}$};
				\end{tikzpicture}$};
			\node at (8,-1) {$$\begin{tikzpicture}[scale=.38, baseline={([yshift=-.8ex]current bounding box.center)}]
				\draw (0,0) rectangle +(1,1);
				\draw (0,1) rectangle +(1,1);
				\draw (1,1) rectangle +(1,1);
				\draw (2,1) rectangle +(1,1);
				\draw (2,2) rectangle +(1,1);
				\draw (2,3) rectangle +(1,1);
				\node at (.5,.5) {$2$};
				\node at (.5,1.5) {$1$};
				\node at (1.5,1.5) {$\overline{2}$};
				\node at (2.5,1.5) {$\overline{1}$};
				\node at (2.5,2.5) {$\overline{3}$};
				\node at (2.5,3.5) {$3$};
				\end{tikzpicture}$$};
			\node at (12,0) {$\SYT{0.15 in}{}{{{3}},{{\overline{3}},{1}},{{\overline{1}},{\overline{2}},{2}}}$};
			\draw [->] (1,.5) -- (3,1);
			\draw [->] (1,-.5) -- (3,-1);
			\draw [->] (5,1) -- (7,1);
			\draw [->] (5,-1) -- (7,-1);
			\draw [->] (9,1) -- (11,.5);
			\draw [->] (9,-1) -- (11,-.5);	
			\end{tikzpicture}
			\vspace{-.3cm}
				
The right key of $T$ has as columns $\YT{0.15in}{}{
			{{3}},
			{{\overline{3}}},
			{{\overline{1}}}}\!r$, $\YT{0.15in}{}{
			{{3}},
		{{\overline{1}}}}\!r$ and $\YT{0.15in}{}{
		{{\overline{1}}}}\!r$, hence $(T)K_+=\YT{0.15in}{}{
			{{3},{3},{\overline{1}}},
			{{\overline{2}},{\overline{1}}},
			{{\overline{1}}}}$.
%%	Consider the Kashiwara-Nakashima tableau $T=\YT{0.17in}{}{
%%		{{1},{2}},
%%		{{3},{\overline{2}}},
%%		{{\overline{3}}}}$. 
%%	We can obtain from $T$ the skew tableau $\SYT{0.17in}{}{
%%		{{2}},
%%		{{3},{1}},
%%		{{\overline{2}}, {\overline{3}}}}$ via symplectic jeu de taquin. Hence, the right key tableau associated to $T$ has as columns $r\!\YT{0.17in}{}{
%%		{{2}},
%%		{{\overline{2}}}}$ and $r\!\YT{0.17in}{}{
%%		{{2}},
%%		{{3}},
%%		{{\overline{2}}}}$. 
%%	Hence $K_+(T)=\YT{0.17in}{}{
%%		{{2},{2}},
%%		{{3},{\overline{1}}},
%%		{{\overline{1}}}}$.

The left key of
 $T$ has as columns
$\YT{0.15in}{}{
	{{1}},
	{{3}},
	{{\overline{3}}}}\!\ell$, $\YT{0.15in}{}{
	{{1}},
	{{2}}}\!\ell$ and $\YT{0.15in}{}{
	{{2}}}\!\ell$, hence $(T)K_-=\YT{0.15in}{}{
	{{1},{1},{2}},
	{{2},{2}},
	{{\overline{3}}}}$.
\end{example}
\subsection{Demazure crystals and right key tableaux}
Let $\lambda\in\Z^n$ be a partition and $v\in\lambda B_n$.
We define $(v)\mathfrak{U}=\{T\in \mathcal{KN}(\lambda,n)\mid (T)K_+=(v)K\}$ the set of KN tableaux of $B^\lambda$ with right key $(v)K$.
Given a subset $X$ of $\mathfrak{B}^\lambda$, consider the operator $\mathfrak{D}_i$ on $X$, $i\in [n]$, defined by
$X\mathfrak{D}_i=\{x\in \mathfrak{B}^\lambda\mid (x)e_i^k\in X \,\,\text{for some $k\geq 0$}\}$ \cite{BSch 17}.
If $v=\lambda \sigma $ where $\sigma =s_{i_1}\dots s_{i_{(\sigma)\ell}}\in B_n$ is a reduced word, we define the Demazure crystal to be
%\begin{align*}%\label{Dem}
$\mathfrak{B}_v=\{(\lambda)K\}\mathfrak{D}_{i_1}\dots \mathfrak{D}_{i_{(\sigma)\ell}}$.
%\end{align*} %Since $(x)e_i^0=x$, we have that if $\rho\leq\sigma$ then $\mathfrak{B}_{\lambda \rho}\subseteq\mathfrak{B}_{\lambda \sigma}$.
 This definition is independent of the reduced word for $\sigma$ \cite[Theorem 13.5]{BSch 17}, %In particular, when $\sigma$ is the longest element of $B_n$ we recover $\mathfrak{B}^{\lambda}$.
  as well of the coset representative of $W_\lambda\sigma$, that is, $\mathfrak{B}_{\lambda \sigma }=\mathfrak{B}_{\lambda \sigma_v }$. %From \cite[Proposition 2.4.4]{BB05}, $\sigma$ uniquely factorizes  as $\sigma'\sigma_v$ where $\sigma'\in W_\lambda$ and $\ell(\sigma)=\ell(\sigma')+\ell(\sigma_v)$. From the signature rule, Subsection \ref{crystal}, if $ \sigma'=s_{j_1}\dots s_{j_{\ell(\sigma')}}\in W_\lambda$ is a reduced word, $\mathfrak{B}_{\lambda \sigma'}=\mathfrak{B}_{\lambda }=\{K(\lambda)\}\mathfrak{D}_{i_1}\dots \mathfrak{D}_{i_{\ell(\sigma')}}=\{K(\lambda)\}$ and we may write in \eqref{Dem} $\mathfrak{B}_{\lambda \sigma}=\mathfrak{B}_{v}$.	
	From \cite[Proposition 2.5.1]{BB05}, if $\rho\leq\sigma$ in $B_n$ then $\rho_u=\sigma_u\leq\sigma_v$ where $u=\lambda\rho$. Since $(x)e_i^0=x$, if $\rho\leq\sigma$ then  $\mathfrak{B}_{\lambda \rho}=\mathfrak{B}_{\lambda \rho_u}\subseteq\mathfrak{B}_{\lambda \sigma_v}=\mathfrak{B}_{ v}$.
	Thus we define the Demazure crystal  atom $\hat{\mathfrak{B}}_{ v}$ to be
	%\begin{align*}
	$\hat{\mathfrak{B}}_{ v}=\hat{\mathfrak{B}}_{ \lambda \sigma}:=\mathfrak{B}_{\lambda \sigma_v }\setminus\bigcup\limits_{\rho_u<\sigma_v }\mathfrak{B}_{\lambda \rho_u}=\mathfrak{B}_{v }\setminus\bigcup\limits_{u<v }\mathfrak{B}_{u}\stackrel{\text{Thm \ref{keycoset}}}{=}\mathfrak{B}_{v }\setminus\bigcup\limits_{(u)K<(v)K }\mathfrak{B}_{u}.$%\end{align*}
%with $\rho< \sigma$ in the Bruhat order.
%
%where the two rightmost identities follow from Theorem \ref{keycoset}.
%
%We define the Demazure crystal atom to be $\hat{\mathfrak{B}}_{\lambda \sigma }:=\mathfrak{B}_{\lambda \sigma }\setminus\bigcup\limits_{\rho<\sigma }\mathfrak{B}_{\lambda \rho}$, with $\rho< \sigma$ in the Bruhat order.
\begin{lemma}\label{fi}
	Let $\sigma=s_i$  be a generator of $B_n$ and $C$ an admissible column such that $(C)f_i\neq 0$. Then $(Cr)\text{\emph{wt}}=(((C)f_i)r)\text{\emph{wt}}$ or   $(Cr)\text{\emph{wt}}=((((C)f_i)r)\text{\emph{wt}})\sigma$.
\end{lemma}

In the previous lemma, all the cases in which the weight is preserved either have equal weight for $i$ and $i+1$ in $Cr$ or $(C)e_i\neq 0$.
Hence we have the following corollaries:

\begin{corollary}\label{onde está fi de T}
	Let $T\in\mathcal{KN}(\lambda, n)$ and $i\in [n]$. If
	$(T)K_+\!=\!(v)K$, for some $v\!=\!(v_1,\dots, v_n)\in\Z^n$, then $((T)f_i)K_+\!=\!(v)K \text{ or } ((T)f_i)K_+=(vs_i)K.$
	Moreover, $((T)f_i)K_+=(vs_i)K$ only if $v_i>v_{i+1}$ (in the usual ordering of real numbers) and $1\leq i<n$, or, $v_i>0$ and $i=n$. 
\end{corollary}

\begin{corollary}\label{ei}
	Let $\sigma=s_i$  be a generator of $B_n$ and $C$ an admissible column. Then $(Cr)\text{\emph{wt}}=(((C)e_i)r)\text{\emph{wt}}$ or $(Cr)\text{\emph{wt}}=((((C)e_i)r)\text{\emph{wt}})\sigma$.
\end{corollary}
%\begin{proof}
%	Call $C'$ to $(C)e_i$. By Lemma \ref{fi} we have that  $wt(rC')=wt(r((C')f_i))$ or   $wt(C')=wt(r((C')f_i))\sigma$, so we have that $wt(r((C)e_i))=wt(rC)$ or   $wt((C)e_i)=wt(rC)\sigma\Leftrightarrow wt((C)e_i)\sigma=wt(rC)$.
%\end{proof}

\begin{lemma}\label{eicheck}
	Let $i\in \left[n\right]$ and $C$ be an admissible column. Then $(C)e_i\neq 0$ only if:
%	\begin{enumerate}
		
		1. $i<n$ and the weight of $i$ in $Cr$ is less than the weight of $i+1$ in $Cr$;
		
		2. $i=n$ and weight of $i$ is negative in $Cr$.
	%\end{enumerate}
%	then we can apply $e_i$ to $C$ (in the sense $(C)e_i\neq 0$).
\end{lemma}
%\begin{proof}
%	If $i=n$ then $-n$ appears on $rC$ and $n$ does not. Since $n$ is the biggest unbarred letter of the alphabet we have that $-n$ also appears in $C$ and $n$ does not. Hence we can apply $e_n$ to $C$.
%	
%	If $i<n$ and the weight of $i$ in $rC$ is less than the weight of $i+1$ in $rC$ then the weight of both can be one of the following three options: $(0,1)$, $(-1,1)$, $(-1,0)$. Note that $rC$ does not have symmetric entries. So in the first two cases we have that $i+1$ exists in $rC$ and $i$ does not, hence $i+1$ exists in $C$ and $i$ does not, so we can apply $e_i$ to $C$.
%	In the last case, we have that $\overline{i}$ exists in $rC$ and $i+1$ and $\overline{i+1}$ does not. Hence we have that $\overline{i}$ exists in $C$ and $i$ or $\overline{i+1}$ does not, so we can apply $e_i$ to $C$.
%\end{proof}
%The next theorem is the main theorem of this paper. It gives a description of a Demazure crystal atom in type $C$ using the right key map (Theorem \ref{rightkeymap}). Lascoux and Schützenberger, in \cite[Theorem 3.8]{LasSchu 90}, proved the type $A$ version of this theorem, where $v\in \N^n$ and, consequently, $\sigma_v\in \mathfrak{S}_n$. For inductive reasoning, used in what follows, we recall the chain property on the set of minimal length coset representatives modulo $W_\lambda$ \cite[Theorem 2.5.5]{BB05}.
Thanks to the right key map, Theorem \ref{rightkeymap}, we now describe the Demazure  crystal atom in type $C$.
%using the right key map (Theorem \ref{rightkeymap}).
Lascoux and Sch\"{u}tzenberger,  \cite[Theorem 3.8]{LasSchu 90}, have given the type $A$ version.
\begin{theorem}[Main Theorem]\label{DemKey}
	Let $v\in \lambda B_n$.
	Then $(v)\mathfrak{U}=\hat{\mathfrak{B}}_{v}$.
\end{theorem}
\begin{proof} Let $\rho$ be a minimal length coset representative modulo $W_\lambda$ such that $v=\lambda\rho$. We will proceed by induction on $(\rho)\ell$. 
	If $(\rho)\ell=0$ then $\rho=id$ and $v=\lambda$. In this case we have that $\widehat{\mathfrak{B}}_\lambda=\{(\lambda)K\}=(\lambda)\mathfrak{U}$.
	
	Let $\rho\geq 0$. Consider $\sigma=s_i$ a generator of $B_n$ such that $\rho\sigma>\rho$ and $\lambda\rho\sigma\neq\lambda\rho=v$, i.e., $\rho\sigma\rho^{-1}\notin W_\lambda$.
	Recall $e_i$, $\varepsilon_i$,  $f_i$ and $\phi_i$ from the definition of the crystal $\mathfrak{B}^\lambda$.
	If $T\in \hat{\mathfrak{B}}_{\lambda\rho\sigma}$ then $T$ is obtained after applying $f_i$ (maybe more than once) to a tableau in $\hat{\mathfrak{B}}_{\lambda\rho}$, which by inductive hypothesis exists in $(v)\mathfrak{U}$. By Corollary \ref{onde está fi de T}, if $(T)f_i\notin(v)\mathfrak{U}$ then $(T)f_i\in (v\sigma)\mathfrak{U}$. So it is enough to prove that given a tableau $T\in (v)\mathfrak{U}\cup (v\sigma)\mathfrak{U}$ then $(T)e_i^{(T)\varepsilon_i}\in (v)\mathfrak{U}$. 
	
	We have two different cases to consider: $i=n$ and $i<n$. %Denote $(T)f_i^{\phi_i(T)}$ by $T^\sigma$.
	
	%As frank words de um associadas a um tableau $T$ vão ser as cr dos skew tableaux relevantes na definiçao de $K(T)$.

%	If $T\in \mathfrak{U}(v)$ and $T^\sigma \notin \mathfrak{U}(v)$ (thus $T^\sigma\in \mathfrak{U}(v\sigma)$) then, if $i<n$, exists a frank word of $T^\sigma$ such that, if $V_1$ is its first column then $rV_1$ has less weight for $i$ than for $i+1$ (less in the of the usual ordering of real numbers); if $i=n$ exists a frank word of $T^\sigma$ such that, if $V_1$ is its first column then $rV_1$ has negative weight for $i$. Since we are in the column $rV_1$, if $i<n$, $i$ and $i+1$ can have weights $(0,1)$, $(-1,1)$ ou $(-1,0)$ and if $i=n$ then $i$ has weight $-1$.

	If $T\in (v\sigma)\mathfrak{U}$ then, if $i<n$, there exists a frank word of $T$ such that, if $V_1$ is its first column then $V_1r$ has less weight for $i$ than for $i+1$ (less in the usual ordering of real numbers); if $i=n$, there is a frank word of $T$ such that, if $V_1$ is its first column then $V_1r$ has negative weight for $i$. Since we are in the column $V_1r$, if $i<n$, $i$ and $i+1$ can have weights $(0,1)$, $(-1,1)$ or $(-1,0)$ and if $i=n$ then $i$ has weight $-1$. Note that these are the exact conditions of Lemma \ref{eicheck}. In either case, due to Lemma \ref{eicheck}, we can applying $e_i$ enough times to the frank word associated until this no longer happens. This is true because we only need to look to $V_1$ to see if it changes after applying $e_i$ enough times to the frank word.
	In the signature rule we have that successive applications of $e_i$ changes the letters of a word from the end to the beginning, so, from the remark after Lemma \ref{fi}, the number of times that we need to apply $e_i$, in order to conditions of Lemma \ref{eicheck} do not hold for the first column, is $(T)\varepsilon_i$. So $\left((T)e_i^{(T)\varepsilon_i}\right)K_+\neq (v\sigma)K$, hence, from Corollary \ref{ei}, we have that $(T)e_i^{(T)\varepsilon_i}\in(v)\mathfrak{U}$.

	If $T\in (v)\mathfrak{U}$ then $(T)e_i^{(T)\varepsilon_i}\in(v)\mathfrak{U}$ because if not, $(T)e_i^{(T)\varepsilon_i}$ will be in a Demazure crystal associated to $\rho'\in B_n$, with $\rho'<\rho$ such that $\rho'\sigma=\rho$. This cannot happen because in this case $\rho'=\rho\sigma<\rho$, which is a contradiction.
\end{proof}
%AAAAAAAAAAAAA desenhar isto
\begin{example}\label{cristal21}
Consider the crystal graph of $\mathfrak{B}^{(2,1)}$:
	\begin{multicols}{2}
		${%\scriptsize \begin{array}{cccc}
			\begin{tikzpicture}
			[scale=.55,auto=left]
			\node (n0) at (0,14.5) {$\YT{0.17 in}{}{{{1},{1}},{{2}}}$};
			\node (n1l) at (-3.5,13.5)  {$\YT{0.17 in}{}{{{1},{2}},{{2}}}$};
			\node (n1r) at (3.5,13.5)  {$\YT{0.17 in}{}{{{1},{1}},{{\overline{2}}}}$};
			\node (n2l) at (-3.5,11)  {$\YT{0.17 in}{}{{{1},{\overline{2}}},{{2}}}$};
			\node (n2r) at (3.5,11) {$\YT{0.17 in}{}{{{1},{2}},{{\overline{2}}}}$};
			\node (n3r) at (3.5,8.5)  {$\YT{0.17 in}{}{{{2},{2}},{{\overline{2}}}}$};
			\node (n4rr) at (5,6)  {$\YT{0.17 in}{}{{{2},{2}},{{\overline{1}}}}$};
			\node (n4r) at (2,6)  {$\YT{0.17 in}{}{{{2},{\overline{2}}},{{\overline{2}}}}$};
			\node (n5r) at (3.5,3.5)  {$\YT{0.17 in}{}{{{2},{\overline{2}}},{{\overline{1}}}}$};
			\node (n6r) at (3.5,1)  {$\YT{0.17 in}{}{{{\overline{2}},{\overline{2}}},{{\overline{1}}}}$};
			\node (n3ll) at (-5,8.5)  {$\YT{0.17 in}{}{{{1},{\overline{2}}},{{\overline{2}}}}$};
			\node (n4l) at (-3.5,6)  {$\YT{0.17 in}{}{{{1},{\overline{1}}},{{\overline{2}}}}$};
			\node (n3l) at (-2,8.5)  {$\YT{0.17 in}{}{{{1},{\overline{1}}},{{2}}}$};
			\node (n5l) at (-3.5,3.5)  {$\YT{0.17 in}{}{{{2},{\overline{1}}},{{\overline{2}}}}$};
			\node (n6l) at (-3.5,1)  {$\YT{0.17 in}{}{{{2},{\overline{1}}},{{\overline{1}}}}$};
			\node (n7) at (0,0)  {$\YT{0.17 in}{}{{{\overline{2}},{\overline{1}}},{{\overline{1}}}}$};
			
			\draw (0,12.5)--(-3.5,15.5);
			\draw (0,12.5)--(3.5,15.5);
			%		\draw (0,12.5)--(1,10);
			\draw (0,12.5)--(5,3.5);
			\draw (0,12.5)--(-5,7);
			%		\draw (0,12.5)--(4,8);
			\draw (0,12.5)--(-2,-.5);
			\draw (0,12.5)--(2,-.5);
			\draw (0,12.5)--(-5,12.5);
			\draw (0,12.5)--(5,12.5);
			%			\draw[->] [draw=red,  thick] ()--();
			%			\foreach \from/\to in {n0/n1r,n6l/n7,n3l/n4l,  n2l/n3ll,n1l/n2l,n3r/n4r,n4rr/n5r,n5r/n6r}
			%			\draw[->] [draw=red,  thick](\from) -- (\to);
			%%%			
			%			\foreach \from/\to in {n0/n1l,n1r/n2r,n2r/n3r,n3r/n4rr,n4r/n5r,n6r/n7,n3ll/n4l,n5l/n6l,n4l/n5l,n2l/n3l}
			%			\draw [->] [draw=blue,  thick](\from) -- (\to);
			\draw[->] [draw=red,  thick] (-2.5,0.7)--(-1,0.3);
			\draw[->] [draw=blue,  thick] (2.5,0.7)--(1,0.3);
			\draw[->] [draw=blue,  thick] (-3.5,2.5)--(-3.5,2);
			\draw[->] [draw=red,  thick] (3.5,2.5)--(3.5,2);
			\draw[->] [draw=blue,  thick] (-3.5,5)--(-3.5,4.5);
			\draw[->] [draw=blue,  thick] (2.5,5)--(3,4.5);
			\draw[->] [draw=red,  thick] (4.5,5)--(4,4.5);
			\draw[->] [draw=red,  thick] (-2.5,7.5)--(-3,7);
			\draw[->] [draw=blue,  thick] (-4.5,7.5)--(-4,7);
			\draw[<-] [draw=blue,  thick] (4.5,7)--(4,7.5);
			\draw[<-] [draw=red,  thick] (2.5,7)--(3,7.5);
			\draw[->] [draw=blue,  thick] (-3,10)--(-2.5,9.5);
			\draw[->] [draw=red,  thick] (-4,10)--(-4.5,9.5);
			\draw[->] [draw=blue,  thick] (3.5,10)--(3.5,9.5);
			\draw[->] [draw=red,  thick] (-3.5,12.5)--(-3.5,12);
			\draw[->] [draw=blue,  thick] (3.5,12.5)--(3.5,12);
			\draw[->] [draw=blue,  thick] (-1,14.2)--(-2.5,13.8);
			\draw[->] [draw=red,  thick] (1,14.2)--(2.5,13.8);
			\end{tikzpicture}
			%\end{array}
		}$
The crystal is split into pieces. Each piece is a Demazure atom and contains exactly one symplectic key tableau, so we can identify each part with the weight of that key tableau, a vector in the $B_2$-orbit of $(2,1)$. From the previous theorem we have that all tableaux in the same piece have the same right key.\linebreak
One can check that $((1,\overline{2}))\mathfrak{U}=$ $\left\{\YT{0.15 in}{}{{{1},{\overline{2}}},{{2}}}, \YT{0.15 in}{}{{{1},{\overline{2}}},{{\overline{2}}}}\right\}=\hat{\mathfrak{B}}_{\lambda s_1s_2}$, for example.
\end{multicols}
\end{example}
%\subsection{Combinatorial description of type $C$ Demazure characters and atoms}
%\vspace{-.5cm}
Given $v\in \lambda B_n$ define the Demazure character (or key polynomial), $\kappa_v$, and the Demazure atom in type $C$, $\widehat{\kappa}_v$, as the generating functions of the KN tableau weights in $\mathfrak{B}_v$ and $\widehat{\mathfrak{B}}_v$, respectively:  $\kappa_v=\sum\limits_{T\in \mathfrak{B}_{v}}x^{(T)\text{wt}},\,
\hat{\kappa}_v=\sum\limits_{T\in \hat{\mathfrak{B}}_{v}}x^{(T)\text{wt}}$. Theorem \ref{DemKey} detects the KN tableaux in $\mathfrak{B}^\lambda$ contributing to the Demazure atom $\hat{\kappa}_v=\sum\limits_{\substack{(T)K_+=(v)K\\
		T\in \mathfrak{B}^\lambda }}x^{(T)\text{wt}}$.
Moreover, one has
		$\kappa_v=\sum\limits_{u\leq v}\hat{\kappa}_u=\sum\limits_{\substack{u\leq v\\T\in (u)\mathfrak{U}}}x^{(T)\text{wt}}=\sum\limits_{\substack{(u)K\leq (v)K\\T\in (u)\mathfrak{U}}}x^{(T)\text{wt}}=\sum\limits_{(T)K_+\leq (v)K}x^{(T)\text{wt}}$.
In Example \ref{cristal21}, $\mathfrak{B}_{(1,\overline{2})}=\{T\in\mathfrak{B}^\lambda\mid (T)K_+\leq ((1,\overline{2}))K\}=\left\{\YT{0.15 in}{}{{{1},{1}},{{2}}}, \YT{0.15 in}{}{{{1},{2}},{{2}}}, \YT{0.15 in}{}{{{1},{1}},{{\overline{2}}}}, \YT{0.15 in}{}{{{1},{\overline{2}}},{{2}}}, \YT{0.15 in}{}{{{1},{\overline{2}}},{{\overline{2}}}}\right\}$.
\section{Lusztig involution, right and left keys}\label{SecLusztig}
 %In type $A$ the Weyl group associated to it is the symmetric group and $\omega_0$ can be seen as the antidiagonal matrix and in type $C$ the Weyl group associated to it is the hyperoctahedral group and $\omega_0$ can be seen as $-\text{Id}$.
Let $\mathfrak{B}^\lambda$ be the crystal of tableaux in $\mathcal{KN}(\lambda, n)$.
	The type $C_n$ Lusztig involution $L:\mathfrak{B}^\lambda\rightarrow\mathfrak{B}^\lambda$ is the only involution such that, for all $i\in [n]$, $x\in \mathcal{KN}(\lambda, n)$:
(1) $((x)L)\text{wt}=((x)\text{wt})\omega_0=-(T)\text{wt}$, where $\omega_0$ is the longest element of $B_n$;
	(2) $((x)L)e_i=((x)f_{i})L$ and $((x)L)f_i=((x)e_{i})L$; % where $i'$ is such that $(\alpha_i)\omega_0=-\alpha_{i'}$. 	
	and (3) $((x)L)\varepsilon_i=(x)\varphi_{i}$ and $((x)L)\varphi_i=(x)\varepsilon_{i}$.%\end{enumerate}
%\end{definition}

%Note that $\omega_0=-\text{Id}$, where $\text{Id}$ is the identity map.
The involution $L$ flips the crystal upside down.
%\begin{definition} \cite{BSch 17}
%	Let $\mathfrak{C}$ be a  connected component  in the crystal $G_n$ of type C, the crystal of words in $[\pm n]^*$. The dual crystal $\mathfrak{C}^\vee$  is the crystal obtained from $\mathfrak{C}$ after reversing the direction of all arrows. Also, the if $x\in\mathfrak{C}$, then for its correspondent in $\mathfrak{C}^\vee$, $x^\vee$, we have $wt(x)=-wt(x^\vee)$.
%\end{definition}
%In type $C$, since $\omega_0=-\text{Id}$, it follows from the definition that $\mathfrak{C}$ and $\mathfrak{C}^\vee$, as crystals in $G_n$, have the same highest weight. Therefore, they are isomorphic. In the case of $\mathfrak{B}^\lambda$, the crystal of KN tableaux of shape $\lambda$,   the Lusztig involution  is  a realization of the dual crystal of $\mathfrak{B}^\lambda$. Hence the crystal $\mathfrak{B}^\lambda$ is self-dual. We shall see other realizations of the dual.
  %\subsection{Evacuation algorithm}
 The Schützenberger evacuation is a realization of Lusztig involution in type $A$. The algorithm below adapts it to KN tableaux. It sends a tableau $T\in \mathcal{KN}(\lambda,n)$ to $T^{\text{Ev}}\in \mathcal{KN}(\lambda,n)$, where $ (T)\text{wt}=-((T^{\text{Ev}})\text{wt})$.
\begin{algorithm}\label{rsi}\indent
		1. Let $(T)cr^\star$ be the word obtained by applying $\omega_0$ to the letters of $(T)cr$ and writing it backwards (or define $T^\#$ by $\pi$-rotating $T$ and applying $\omega_0$ to its entries).
		
		2. Insert $(T)cr^\star$ (or rectify $T^\#$). Define $T^{Ev}:=((T)cr^\star)P=$rectification of $T^\#$.
%	\end{enumerate}
\end{algorithm}
%\begin{example}
Consider the KN tableau $T=\YT{0.15in}{}{
			{{1},{\overline{2}}},
			{{\overline{3}},{\overline{1}}},
			{{\overline{2}}}}$. 
		Then, $w=(T)cr=\overline{2}\overline{1}1\overline{3}\overline{2}$ and $w^\star=23\overline{1}12$.
		Then
%the following sequence of tableaux:
%		$\YT{0.17in}{}{
%			{{{2}}}}
%		\YT{0.17in}{}{
%			{{2}},
%			{{3}}}
%		\YT{0.17in}{}{
%			{{2}},
%			{{3}},
%			{{\overline{1}}}}
%		\YT{0.17in}{}{
%			{{2},{2}},
%			{{3}},
%			{{\overline{2}}}}$
		$(w^\star)P=\YT{0.15in}{}{
			{{2},{2}},
			{{3},{3}},
			{{\overline{3}}}}$ which is the rectification of $T^\#=\SYT{0.15in}{}{
			{{2}},
			{{3},{1}},
			{{\overline{1}},{2}}}$. 
Let $w\in [\pm n]$. The connected component of the crystal $G_n$ that contains the word $w^\star $ is obtained applying $^\star$ to each vertex of the one containing $w$ and reverting arrows. They are  isomorphic  because they have the same highest weight, say $\lambda$.
Therefore $(w)P$, $(w^\star)P\in \mathfrak{B}^\lambda$ and have symmetric weights. $T^{Ev}$ is the only KN tableau with the same shape as $T$, and Knuth equivalent to $(T)cr^\star$.
The algorithm  is a realization of the type $C$ Lusztig involution.
It follows that  evacuation of the right key of a tableau is the left key of the evacuation of the same tableau. See \cite{San 19} for the proof.
\begin{theorem}\label{key&Lusz}
	Let T be a KN tableau and $^{Ev}$ the type $C$ Lusztig (Schützenberger) involution. Then
	$$(T)K_+\,^{Ev}=(T^{Ev})K_-.$$
\end{theorem}
\section{Acknowledgements}
I am grateful to C. Lenart, for his observations on the alcove path model, and to O. Azenhas, my Ph.D. advisor, for her help on the preparation of this paper.
\vspace{-.3cm}
\footnotesize
\bibliographystyle{amsplain}
\providecommand{\bysame}{\leavevmode\hbox to3em{\hrulefill}\thinspace}

\end{document}